\documentclass[11pt]{article}
\usepackage{amsmath, amssymb, amsthm, amsfonts}
\usepackage{indentfirst}
\usepackage[mathscr]{euscript}
\usepackage{amscd}
\usepackage{color}
\usepackage[pdftex]{graphicx}

\numberwithin{equation}{section}
\numberwithin{figure}{section}

\theoremstyle{plain}
\newtheorem{theorem}{Theorem}\numberwithin{theorem}{section}
\newtheorem{lemma}{Lemma}\numberwithin{lemma}{section}
\numberwithin{proposition}{section}
\newtheorem{corollary}{Corollary}\numberwithin{corollary}{section}
\theoremstyle{definition}
\numberwithin{definition}{section}
\theoremstyle{remark}
\newtheorem{remark}{Remark}\numberwithin{remark}{section}

\newcommand{\R}{\mathbb{R}}

\newcommand{\po}{\mathcal{P}_{1, 2}}
\newcommand{\J}{\mathbb{J}}
\newcommand{\F}{{}_1F_2}

\title{Newton diagram of positivity for \\$\F$ generalized hypergeometric functions}

\author{Yong-Kum Cho and Hera Yun}

\date{}

\begin{document}

\maketitle

\bigskip

\noindent
{\bf Abstract.} As for the positivity of $\F$ generalized hypergeometric functions,
we present a list of necessary and sufficient conditions in terms of parameters and determine the region of
positivity by certain Newton diagram.

\bigskip


\bigskip


\section{Introduction}
In this paper we are concerned with the problem of positivity for the
generalized hypergeometric functions of type
\begin{equation}\label{I1}
\F\left(a\,; b, c\,; \,-\frac{\,x^2}{4}\right)
= \sum_{k=0}^\infty \frac{(a)_k}{k! (b)_k (c)_k}\left(-\frac {\,x^2}{4}\right)^{k}\qquad(x\in\R),
\end{equation}
where the parameters $\,a, b, c\,$ are positive and the coefficients are written
in Pochhammer's symbols, that is,
$\,(\alpha)_k = \Gamma(\alpha +k)/\Gamma(\alpha)\,$ for any $\,\alpha>0.$

In view of symmetry about zero, it will be assumed throughout $\,x>0.\,$ To simplify our notation,
we shall denote by $\po$ the set of all triples $\,(a, b, c)\,$ for which the corresponding generalized
hypergeometric functions are strictly positive on the interval $(0, \infty)$, that is,
\begin{equation}\label{I2}
\mathcal{P}_{1, 2} = \left\{(a, b, c)\in\R_+^3:
\F\left(a\,; b, c\,; -\frac{\,x^2}{4}\right)>0\,\right\}\,.
\end{equation}

Due to a variety of applications, the problem has been a historic issue. In classical analysis,
it arises frequently in connection with the integrals or sums of those special functions
expressible in terms of the generalized hypergeometric functions of type \eqref{I1}.
While we refer to Askey \cite{As2}, Fields and Ismail \cite{FI},
Gasper \cite{Ga1} for further backgrounds and references, we state some of earlier results
relevant to the present work.

\begin{itemize}
\item[(i)] For $\,\alpha>-1,\,$ if $J_\alpha$ stands for the Bessel function of order $\alpha$,
then
\begin{align}\label{R0}
\int_0^x J_\alpha(t) t^{-\beta} dt &= \frac{x^{\alpha-\beta+1}}{2^\alpha(\alpha-\beta+1)\Gamma(\alpha+1)}\nonumber\\
&\,\,\times\, \F\left(\frac{\alpha-\beta+1}{2}\,; \alpha+1,\,
\frac{\alpha-\beta+3}{2}\,;\,-\frac{\,x^2}{4}\right),
\end{align}
where $\,\alpha-\beta+1>0\,$ and
$B$ denotes Euler's beta function. For its positivity,
Cooke \cite{C} proved the case $\,\beta=0\,,$ which is equivalent to
\begin{equation}\label{R1}
(a, \,a+1, \,2a)\in\po \qquad(a>0),
\end{equation}
and Makai \cite{M} proved the case $\,\beta= -1/2,\,\alpha>1/2,\,$ that is,
\begin{equation}\label{R2}
\left(a, \, a+1, \, 2a - \frac 12\right)\in\po \qquad(a>1).
\end{equation}
\item[(ii)] In connection with completely monotone functions of certain type and
the positivity of Ces\'aro means of Jacobi series, Askey and Pollard \cite{AP} and
Fields and Ismail \cite{FI} proved separately
\begin{align}
\left(a, \, 2a, \, 2a + \frac 12\right) &\in\po\qquad(a>0),\label{R3}\\
\left(a, \, \rho a, \, \rho a + \frac 12\right) &\in\po \qquad (a>1,\,\rho\ge 3/2).\label{R4}
\end{align}
\item[(iii)] Regarding the sign of Lommel's function, Steinig \cite{St} proved
\begin{equation}\label{R5}
\left(1,\, b, \,\frac 72 -b\right)\in\po\quad\left(3/2<b<2\right).
\end{equation}
\item[(iv)] By considering certain fractional integral transforms with
the squares of Bessel functions as kernels, Gasper \cite{Ga1} established
\begin{equation}\label{R6}
\left(a, \,a+ \frac 12 + \delta,\,2a\right)\in\po\qquad(a>0,\,\delta>0),
\end{equation}
which includes \eqref{R1}, \eqref{R3} as a special case $\,\delta = 1/2,\,a,\,$ respectively.
In addition, as it will be explained below,
Gasper invented a series expansion method for investigating positivity and obtained a number of positivity
results for the Bessel integrals of certain type.
\end{itemize}

Our aim here is to determine the general patterns of parameters for positivity instead of
special cases and provide information on the location of the first zeros when positivity breaks down.
Our approach is based on Gasper's method from a slightly different viewpoint and integral transforms
with the squares of Bessel functions as its kernels.

To describe briefly our main result, we shall decompose the first quadrant into three disjoint regions
$\,\R_+^2 = P_a\cup N_a\cup O_a\,$ for each $\,a>0\,$ and prove that if $\,(b, c)\in P_a\,,$ then
$\,(a, b, c)\in\po\,$ and if $\,(b, c)\in N_a\,,$ then the corresponding generalized hypergeometric function
alternates in sign. It turns out that the positivity region $P_a$ coincides with certain Newton diagram
and the missing region $O_a$ consists of two strips cut by the line of necessity.

\section{Preliminaries}
The generalized hypergeometric functions of type \eqref{I1} includes a special
class of Bessel functions defined by
\begin{align}\label{N1}
\J_\alpha(x) = {}_0F_1\left(\alpha +1\,; -\frac{\,x^2}{4}\right)
= \Gamma(\alpha+1)\left(\frac x2\right)^{-\alpha}J_\alpha(x)
\end{align}
for $\,\alpha>-1,\,$ which will play fundamental roles in what follows.

As $\J_\alpha$ and $J_\alpha$ share positive zeros in common, it follows readily
from the theory of Bessel functions (see \cite{E}, \cite{L}, \cite{Wa}) that $\J_\alpha$ possesses
infinitely many positive zeros all of which are simple and the zeros of $\J_\alpha$ and $\J_{\alpha+1}$ are interlaced.
Moreover, if $j_\alpha$ denotes the smallest positive zero of $J_\alpha$, then
\begin{align}
\frac{(3+2\alpha)\pi}{4} &<j_\alpha<\frac{(7+2\alpha)\pi}{8}\quad\text{for}\quad |\alpha|<1/2,\nonumber\\
\sqrt{\alpha(\alpha +2)\,} &< j_\alpha < \sqrt{2(\alpha +1)(\alpha +3)\,}\quad\text{for}\quad \alpha>0.
\end{align}

If $2\alpha$ is an odd integer, then $\J_\alpha$ is expressible in terms
of algebraic and trigonometric functions via the recurrence relation
\begin{align}
&\quad \mathbb{J}_{-\frac 12}(x) = \cos x\,,\quad \mathbb{J}_{\frac 12}(x) = \frac{\sin x}{x},\nonumber\\
&\mathbb{J}_{\alpha+1}(x) = \frac{\,4\alpha(\alpha+1)\,}{x^2}\left[\mathbb{J}_\alpha(x) -\mathbb{J}_{\alpha-1}(x)\right].
\end{align}

Another special class of functions is the squares of $\J_\alpha$ in the form
\begin{align}
\mathbb{J}_\alpha^2\left(\frac x2\right) =\F\left(\alpha + \frac 12\,;\alpha +1, \,2\alpha +1 \,; - \frac{\,x^2}{4}\right)
\end{align}
for $\,\alpha>-1/2,\,$ which will play decisive roles in determining positivity.

\section{Necessity}
It is simple to extend a known positivity result to other parameter cases
on consideration of the fractional integral transform with appropriate kernel. Although it is observed by several authors
(\cite{As2}, \cite{FI}, \cite{Ga1}), we make this basic but important aspect precise under
a slightly weaker condition.

\smallskip

\begin{lemma}\label{lemmaA} For $\,a>0,\,b>0,\,c>0,\,$ suppose that
$$\Phi(x) = \F\left(a\,; b, c\,; -\frac{\,x^2}{4}\right)\ge 0.$$
Then for any $\,0\le\gamma<a,\,\delta\ge 0,\,\epsilon\ge 0,\,$
not simultaneously zero,
$$\,\left(a-\gamma, \,b+\delta, \,c+\epsilon\right)\in
\mathcal{P}_{1, 2}\,.\,$$
\end{lemma}

\begin{proof}
Since the complex extension $\,z\mapsto \Phi(z),\,z\in\mathbb{C},\,$ is obviously entire,
$\Phi$ has only isolated zeros on the interval $(0, \infty)$. For each $\,x>0,\,$ hence, it follows from the
nonnegativity assumption that there exists a non-empty open interval $\,I_x\subset [0, 1]\,$
such that $\,\Phi(xt)>0\,$ for all $\,t\in I_x.\,$

For $\,\delta>0\,$ and $\,x>0,\,$ it is plain to evaluate and deduce
\begin{align*}
\F\left(a\,; b+\delta, \,c\,; -\frac{\,x^2}{4}\right)
&= \frac{2}{B(b, \delta)} \int_0^1 \Phi(xt) (1-t^2)^{\delta-1} t^{2b-1}\,dt\\
&\ge\frac{2}{B(b, \delta)} \int_{I_x}\Phi(xt) (1-t^2)^{\delta-1} t^{2b-1}\,dt\\
&>0.
\end{align*}
In view of the integral representation
\begin{align*}
&\F\left(a\,; b+\delta, \,c+\epsilon\,; -\frac{\,x^2}{4}\right)\nonumber\\
&\quad = \frac{2}{B(c, \epsilon)} \int_0^1 \F\left(a\,; b+\delta, c\,; -\frac{\,x^2t^2}{4}\right)
\,(1-t^2)^{\epsilon-1} t^{2c-1}\,dt
\end{align*}
for $\,\epsilon>0,\,$ it follows from the same analysis or directly from the continuity and positivity
of kernel that $\,(a, b+\delta, c+\epsilon)\in \mathcal{P}_{1, 2}\,$ for any $\,\delta\ge 0,\,\epsilon\ge 0\,$
unless $\,\delta=\epsilon=0.\,$ In the last place, if
$\,0<\gamma<a,\,\delta\ge 0,\,\epsilon\ge 0,\,$ then
\begin{align*}
&\F\left(a-\gamma\,; b+\delta, \,c+\epsilon\,; -\frac{\,x^2}{4}\right)
= \frac{2}{B(a-\gamma, \gamma)} \nonumber\\
&\quad\times\,\int_0^1 \F\left(a\,; b+\delta, c+\epsilon\,; -\frac{\,x^2t^2}{4}\right)
\,(1-t^2)^{\gamma-1} t^{2(a-\gamma)-1}\,dt,
\end{align*}
which implies $\,(a-\gamma, b+\delta, c+\epsilon)\in \mathcal{P}_{1, 2}\,$ by the same reasoning as above.
\end{proof}

With the aid of Lemma \ref{lemmaA}, an inspection on the asymptotic behavior
leads to the following necessary condition for positivity and information on the location of the first zero
on $(0, \infty)$ if violated. We recall that the smallest positive zero of the Bessel function $J_\alpha$ or $\J_\alpha$
is denoted by $j_\alpha$.

\smallskip

\begin{theorem}\label{theoremB0}
For $\,a>0,\,b>0,\,c>0,\,$ put
$$\Phi(x) = \F\left(a\,; b, \,c\,; -\frac{\,x^2}{4}\right)\qquad(x>0).$$
If $\,(a, b, c)\in\po,\,$ then necessarily
$\,b>a,\,c>a,\,b+c \ge 3a + \frac 12\,.$ Moreover, the following properties
on the sign and zeros of $\Phi$ hold true.
\begin{itemize}
\item[\rm(i)] If $\,0<b\le a\,$ or $\,0<c\le a,\,$ then $\Phi$ alternates in sign and has at least one zero on
the interval $\,\left(0, \,j_{c-1}\right]\,$ or $\,\left(0, \,j_{b-1}\right],$ respectively.
\item[\rm(ii)] Suppose that either $\,b\le a +\frac 12\,,\,c\le 2a\,$ or $\,b\le 2a\,,\,c\le a +\frac 12\,$
subject to the condition $\,b+c<3a + \frac 12.\,$
Then $\Phi$ alternates in sign and
has at least one zero on the interval $\,\left(0, \,2j_{a-\frac 12}\right).$
\end{itemize}
\end{theorem}

\begin{proof}
For $\,(a, b, c)\in\po\,,$ if $\,0<b\le a,\,$ then Lemma \ref{lemmaA} implies
$$\F\left(b\,;\,b, c\,; -\frac{\,x^2}{4}\right) = \J_{c-1}(x)>0,$$
which contradicts the fact that the Bessel function $\J_{c-1}$
has infinitely many zeros on the interval $(0, \infty)$. Thus $\,b>a\,$ and
also $\,c>a\,$ by symmetry. In view of the asymptotic behavior (see e.g. Luke \cite{L})
\begin{align}\label{I3}
\Phi(x) &=\frac{\Gamma(b)\Gamma(c)}{\Gamma(b-a)\Gamma(c-a)}
\,\left(\frac x2\right)^{-2a} \left[ 1+ O\left(x^{-2}\right)\right]\nonumber\\
& + \,\,\frac{\Gamma(b)\Gamma(c)}{\sqrt \pi\,\Gamma(a)}
\,\left(\frac x2\right)^{-\sigma} \left[\cos\left( x - \frac{\pi\sigma}{2}\right)
+ O\left(x^{-1}\right)\right]
\end{align}
as $\,x\to\infty,\,$ where $\,\sigma = b+c -a - 1/2,\,$
it is immediate to observe
the condition $\, \sigma\ge 2a\,$ is necessary, that is, $\,b+c\ge 3a + 1/2\,.$

To prove part (i), assume first $\,b=a.\,$ Then $\,\Phi(x) = \J_{c-1}(x)\,$ and the assertion
is trivial. We next assume $\,0<b<a\,$ and represent
\begin{align*}
\J_{c-1}(x) =\frac{2 x^{2-2a}}{B(b, a-b)} \int_0^x\Phi(s) (x^2-s^2)^{a-b-1} s^{2b-1} ds
\end{align*}
for each $\,x>0.$ Evaluating at the first positive zero of $\J_{c-1}$, we have
$$0=\int_0^{j_{c-1}}\Phi(s) \left(j_{c-1}^2 -s^2\right)^{a-b-1} s^{2b-1} ds,$$
which, together with $\,\Phi(0)=1,\,$ implies that $\Phi$ must possess at least one zero on $\,\bigl(0,\,j_{c-1}\bigr)\,$
and alternate in sign. The case $\,0<c\le a\,$ is similar.

To prove part (ii), assume first $\,b=a+ \frac 12\,,\,0<c<2a.$ Since
\begin{align*}
\J_{a-\frac 12}^2\left(\frac x2 \right) &= \F\left(a\,; a+ \frac 12\,,\,2a\,; -\frac{\,x^2}{4}\right)\\
&=\frac{2 x^{2-4a}}{B(c, 2a-c)} \int_0^x\Phi(s) (x^2-s^2)^{2a-c-1} s^{2c-1} ds
\end{align*}
for each $\,x>0,$ we may conclude by the same reasoning as above that $\Phi$ alternates in sign and
has at least one zero on the interval $\,\left(0, \,2j_{a-\frac 12}\right).$
We next assume  $\,0<b<a+ \frac 12\,,\,0<c<2a.$ In this case, we set
$$\Psi(x) = \F\left(a\,; a+ \frac 12\,,\,c\,; -\frac{\,x^2}{4}\right)$$
and deduce the following integral representations
\begin{align*}
\Psi(x) &=
\frac{2 x^{1-2a}}{B\left(b, a+\frac 12 -b\right)} \int_0^x\Phi(s) (x^2-s^2)^{a-b- \frac 12} s^{2b-1} ds,\\
\J_{a-\frac 12}^2\left(\frac x2 \right) &=
\frac{2 x^{2-4a}}{B\left(c, 2a-c\right)}\int_0^x
\Psi(s)(x^2-s^2)^{2a-c- 1} s^{2c-1} ds
\end{align*}
for each $\,x>0.$ Evaluating at $\,x= 2j_{a-\frac 12}\,,$ the second identity implies that $\Psi$ must have
a zero on $\,\left(0, \,2j_{a-\frac 12}\right),$ say, at $\,x_0.$ Evaluating in turn at $\,x= x_0,\,$
the first identity implies that $\Phi$ must alternate in sign and have a zero on $(0, x_0).$ Thus we deduce the same
conclusion as before. In view of symmetry, the proof for part (ii) is now complete.
\end{proof}

\begin{remark}
As our proof and graphical simulations suggest,
it is plausible that {\it if $\,0<b<a,\,$ then $\Phi$ possesses infinitely many zeros and
the zeros of $\Phi$ and $\J_{c-1}$
are interlaced} but we were not able to prove or disprove.
\end{remark}

\medskip

By modifying the proof of Theorem \ref{theoremB0} in an obvious way, it is not hard to deduce the following
counterpart of Lemma \ref{lemmaA} useful in extending a known non-positivity result to other
parameter cases.

\smallskip

\begin{lemma}\label{lemmaC}
For $\,a>0,\,b>0,\,c>0,\,$ put
$$\Phi(x)=\F\left(a\,; b, c\,; -\frac{\,x^2}{4}\right).$$
If $\Phi$ alternates in sign and $z_1$ denotes the first zero on $(0, \infty)$,
then for any $\,\gamma\ge 0,\, 0\le \delta<b, \,0\le \epsilon<c,\,$ the function $\Psi$ defined by
$$\Psi(x)= \F\left(a+\gamma\,; b-\delta, \,c-\epsilon\,; -\frac{\,x^2}{4}\right)$$
also alternates in sign and possesses at least one zero on $(0, z_1].$
\end{lemma}

\section{Integral transform method}
If we take the squares of Bessel functions in the form
$$\F\left(a\,; a+\frac 12\,,\,2a\,; -\frac{\,x^2}{4}\right) = \J_{a-\frac 12}^2
\left(\frac x2\right)\ge 0,$$
then it is straightforward to deduce by Lemma \ref{lemmaA} and Theorem \ref{theoremB0} the following positivity results
for which we shall omit proofs.

\smallskip

\begin{theorem}\label{theoremA}
For $\,a>0,\,$ if $\,b\ge a +\frac 12\,,\,c\ge 2a\,$ or $\,b\ge 2a\,,\,c\ge a +\frac 12\,$
subject to the condition $\,b+c>3a + \frac 12\,,$ then $\,(a, b, c)\in\po\,$
and the following additional properties hold true.
\begin{itemize}
\item[\rm(i)] $\,(a, b, 2a)\in\po\,$ if and only if $\,b>a+ \frac 12\,$ and
\begin{align*}
&\F\left(a\,; a+ \frac 12 +\delta,\,2a\,;
- \frac{\,x^2}{4}\right)\\
&\quad=\frac{2}{B\left(\delta, a+ \frac 12\right)} \int_0^1\J_{a-\frac 12}^2\left(\frac{xt}{2}\right) (1-t^2)^{\delta-1} t^{2a} dt>0\qquad(\delta>0).
\end{align*}
\item[\rm(ii)] $\,\left(a, a+ \frac 12, c\right)\in\po\,$ if and only if $\,c> 2a\,$ and
\begin{align*}
&\F\left(a\,; a+ \frac 12\,,\,2a+\epsilon\,;
- \frac{\,x^2}{4}\right)\\
&\quad= \frac{2}{B\left(\epsilon, 2a\right)} \int_0^1\J_{a-\frac 12}^2\left(\frac{xt}{2}\right) (1-t^2)^{\epsilon-1} t^{4a-1} dt>0\qquad(\epsilon>0).
\end{align*}
\end{itemize}
\end{theorem}

As noted earlier, the sufficiency of (i) owes to Gasper \cite{Ga1}.
Noteworthy are the following special cases.

\paragraph{(A1)} The choice $\,\epsilon = 1-a \,$ in (ii) yields
\begin{align}\label{A1}
&\F\left(a\,; a+ \frac 12\,,\,a+1\,;
- \frac{\,x^2}{4}\right)\nonumber\\
&\qquad= \frac{2}{B\left(1-a, 2a\right)} \int_0^1\J_{a-\frac 12}^2\left(\frac{xt}{2}\right)
\frac{t^{4a-1}\, dt}{(1-t^2)^a}>0 \quad(0<a<1).
\end{align}

As an application, we give an affirmative answer for an
open question of Fields and Ismail \cite{FI} regarding positivity of the generalized
hypergeometric function of the form \eqref{R4}
for $\,3/2<\rho<2\,$ at some value $\,0<a<1.$

\smallskip

\begin{corollary} For $\,\frac 32<\rho<2,\,$ there exists $\,\frac 12<a<1\,$ such that
\begin{align*}
\F\left(a \,;\rho a, \, \rho a + \frac 12\,; -\frac{\,x^2}{4}\right)>0.
\end{align*}
\end{corollary}

\begin{proof} Take $\,a = 1/2(\rho -1)\,$ in \eqref{A1}.
\end{proof}

\paragraph{(A2)} The choice $\,a = 1\,$ in (ii) yields

\begin{align}\label{A2}
&\F\left(1\,; \frac 32\,,\, 2 + \epsilon\,; -\frac{\,x^2}{4}\right)\nonumber\\
&\qquad =\frac{8\epsilon(\epsilon+1)}{x^2}\int_0^1\sin^2\left(\frac{xt}{2}\right) (1- t^2)^{\epsilon-1} t\,dt
 >0 \qquad(\epsilon>0).
\end{align}

A rearrangement of parameters in the form $\,\epsilon +1/2 = \delta\,$
leads to a simple proof of positivity for Struve's function
(cf. Watson \cite{Wa}).

\begin{corollary}\label{corollaryH}
For $\,\delta>\frac 12\,$ and $\,x>0,$ Struve's function is positive
with
\begin{align*}
\mathbf{H}_\delta (x) &= \frac{2\left(x/2\right)^{\delta +1}}{\sqrt \pi\,\Gamma(\delta + 3/2)}\,
 \F\left(1\,; \frac 32,\,\frac 32 +\delta\,; \,-\frac{\,x^2}{4}\right)\nonumber\\
 &= \frac{(2\delta-1)(2\delta+1)\left(x/2\right)^{\delta -1}}{\sqrt \pi\,\Gamma(\delta + 3/2)}
 \int_0^1\sin^2\left(\frac{xt}{2}\right) (1- t^2)^{\delta - \frac 32}\,t\,dt>0.
 \end{align*}
\end{corollary}

\section{Gasper's series method}
In \cite{Ga1} Gasper discovered a series expansion formula for an arbitrary ${}_pF_{p+1}$
generalized hypergeometric function in terms of squares of Bessel functions,
which sets up a remarkably effective method for inspecting positivity
when the signs of coefficients are deterministic.

In a slightly different form, Gasper's formula asserts in effect
\begin{align}\label{G1}
&\F\left(a\,; b, \, c\,;\, -\frac{\,x^2}{4}\right)
=\mathbb{J}_\nu^2\left(\frac x2\right) + R_\nu(x),\nonumber\\
&\quad R_\nu(x) = \sum_{n=1}^\infty \frac{S(n, \nu)}{n!}\frac{2n+2\nu}{n+2\nu}
\frac{(2\nu+1)_n}{(\nu + 1)_n^2}\left(\frac x4\right)^{2n}\J^2_{n+\nu}\left(\frac x2\right),
\end{align}
where $\nu$ is any real such that $2\nu$ is not a negative integer and
\begin{equation}\label{G2}
S(n, \nu) = {}_4F_3\left(\begin{array}{c}
-n, \,\,n+2\nu, \,\,\nu+1, \,\,a\\
\nu +\frac 12, \,\,b, \,\,c \end{array}\biggl| 1
\right).
\end{equation}

In view of the interlacing property on the zeros of Bessel functions,
Gasper's formula \eqref{G1} can be interpreted as follows.
\begin{itemize}
\item[(i)] If the coefficients of $R_\nu$ are all positive, then $\,(a, b, c)\in\po\,$ with
\begin{equation}\label{G3}
\F\left(a\,; b, \, c\,;\, -\frac{\,x^2}{4}\right)> \J_\nu^2\left(\frac x2\right).
\end{equation}
\item[(ii)] If the coefficients of $R_\nu$ are all negative, then
\begin{equation}\label{G4}
\F\left(a\,; b, \, c\,;\, -\frac{\,x^2}{4}\right)<\J_\nu^2\left(\frac x2\right)
\end{equation}
so that it possesses at least one zero on the interval $\left(0, \,2j_{\nu}\right).$
\end{itemize}

To implement Gasper's series method, we shall reduce $S(n, \nu)$ to a ${}_3F_2$ terminating series for which
two types of summation formulas are available. For each integer $n$, Saalsch\"uz's formula asserts
\begin{align}\label{G5}
{}_3F_2\left(\begin{array}{c} -n,\,n+ \alpha,\,\beta\\ \gamma,\,1+\alpha +\beta -\gamma\end{array}
\biggl| 1\right) &= \frac{(\alpha + 1- \gamma)_n (\gamma-\beta)_n}{(\gamma)_n (1+\alpha +\beta -\gamma)_n}
\end{align}
and Watson's formula takes the form
\begin{align}\label{G6}
{}_3F_2\left(\begin{array}{c} -n,\,n+ 2\alpha,\,\beta\\ \alpha + 1/2\,,\,2\beta\end{array}
\biggl| 1\right) = \frac{(1/2)_k (\alpha +1/2 -\beta)_k}{(\alpha+1/2)_k (\beta +1/2)_k}
\end{align}
if $\,n=2k, \,k = 1, 2, \cdots, \,$ and zero otherwise. Both formulas are valid if
the denominators are not equal to zero (see Bailey \cite{B} and Luke \cite{L}).

On account of symmetry in $\,b, c,\,$ there are only three possible ways of reduction.
In the first place, if we take $\, \nu = a - 1/2\,$ and $\, b+c = 3a + 1/2,\,$
the boundary plane of the necessity region for positivity,
then our evaluation of $S(n, \nu)$ by Saalsch\"uz's formula \eqref{G5} gives the series development
\begin{align}\label{G7}
&\F\left(a\,; b, \, 3a + \frac 12 -b \,; -\frac{\,x^2}{4}\right)
\nonumber\\
&\qquad\quad =\J_{a-\frac 12}^2\left(\frac x2\right) +
\sum_{n=1}^\infty \frac{2n + 2a-1}{n+ 2a-1}\frac{(2a)_n}{n!\,(a+ 1/2)_n^2}
\nonumber\\
&\qquad\qquad\quad\times\,\,
\frac{(2a-b)_n \left(b-a- 1/2\right)_n}{(b)_n \left(3a+1/2-b\right)_n}
\left(\frac x4\right)^{2n} \J^2_{n+ a-\frac 12}\left(\frac x2\right).
\end{align}

\smallskip

\begin{theorem}\label{theoremB}
For $\,a>0,\,b>0,\,$ we have
\begin{equation}\label{G8}
\left(a, \,b, \,3a + \frac 12-b\right)\in\mathcal{P}_{1, 2}
\end{equation}
if and only if $\,a\ne 1/2\,$ and $b$ lies strictly between $\,a+1/2\,$ and $2a$. Moreover, the following
inequality holds true in such a case:
\begin{align}\label{G9}
\F\left(a\,; b, \, 3a + \frac 12 -b \,; -\frac{\,x^2}{4}\right)
>\J_{a-\frac 12}^2\left(\frac x2\right)
\end{align}
\end{theorem}

\begin{proof}
In view of the necessity for positivity, it suffices to deal with
the case $\,a<b<2a + 1/2\,.$ An inspection on
Gasper's sum \eqref{G7} reveals the coefficients are all positive if
$\,2a<b<a+ 1/2\,$ for $\,0<a<1/2\,$ or $\,a+ 1/2 <b<2a\,$ for $\,a>1/2\,$ and all negative otherwise.
Hence the assertion follows and the inequality \eqref{G9} is a simple consequence of our interpretation \eqref{G3}.

In the case $\,a=1/2,\,1/2<b<3/2,\,$ \eqref{G7} reduces to
\begin{align*}
&\F\left(\frac 12\,; b, \, 2-b \,; -\frac{\,x^2}{4}\right)
\nonumber\\
&\qquad =\J_{0}^2\left(\frac x2\right) +
 2\sum_{n=1}^\infty
\frac{(1-b)_n \left(b-1\right)_n}{(n!)^2\,(b)_n \left(2-b\right)_n}
\left(\frac x4\right)^{2n} \J^2_{n}\left(\frac x2\right).
\end{align*}
If $\,b\ne 1,\,$ then the coefficients of the series on the right are easily seen to be negative and hence
the function on the left must possess at least one
zero on the interval $\left(0, \,2j_{0}\right).$ If $\,b=1,\,$ then the series vanishes and
$$\F\left(\frac 12\,; 1, \, 1\,; -\frac{\,x^2}{4}\right) = \J_{0}^2\left(\frac x2\right),$$
which possesses infinitely many zeros.
\end{proof}

\begin{remark}\label{remarkB}
In the case when positivity breaks down, it is not hard to figure out the cases of $\,a, b\,$
in which the opposite inequality
\begin{align}\label{G10}
\F\left(a\,; b, \, 3a + \frac 12 -b \,; -\frac{\,x^2}{4}\right)
<\J_{a-\frac 12}^2\left(\frac x2\right)
\end{align}
holds true, e.g., the case $\,a<b<2a,\,0<a<\frac 12\,$ or $\,a<b<a+\frac 12,\,a>\frac 12\,,$
which implies the function on the right alternates in sign and
possesses at least one zero on the interval $\,\left(0, \,2j_{a-\frac 12}\right)$.
\end{remark}

\medskip

As an illustration, we take $\,a=1/8\,$ and $\,b= 3/8\,$ to see
$$\J_{-\frac 38}^2\left(\frac x2\right)<\F\left(\frac 18\,; \frac 38, \, \frac 12 \,; -\frac{\,x^2}{4}\right)
=\Phi_1(x).$$
For the opposite monotonicity, we take $\,b= 3/16\,$ to see
$$\J_{-\frac 38}^2\left(\frac x2\right)>\F\left(\frac 18\,; \frac {3}{16}, \, \frac {11}{16}\,; -\frac{\,x^2}{4}\right)
=\Phi_2(x).$$
In Figure \ref{Fig23}, each curve from the top represents
$\,\Phi_1(x), \,\J_{-\frac 38}^2\left(\frac x2\right),\,\Phi_2(x).\,$

\medskip

\begin{figure}[!h]
 \centering
 \includegraphics[width=300pt, height=200pt]{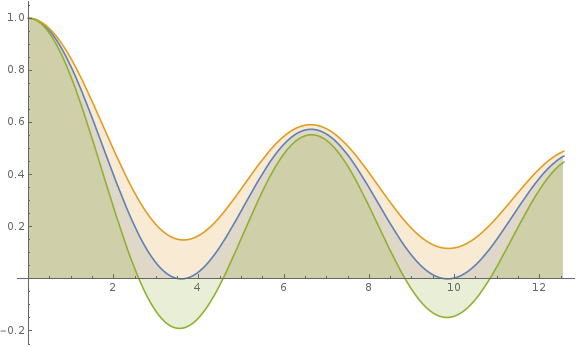}
 \caption{Monotonicity of generalized hypergeometric functions}
\label{Fig23}
\end{figure}

\paragraph{(B1)} If we consider the dilation $\,b= \rho a,\,1<\rho <2,\,$ then
\begin{equation}\label{G11}
\left(a,\, \rho a, \,(3-\rho) a + \frac 12\right)\in \mathcal{P}_{1, 2}
\quad\text{if and only if}\quad a> \frac{1}{2(\rho-1)}\,.
\end{equation}
The special case $\,\rho = 3/2\,$ gives \eqref{R4}, the result of Fields and Ismail,
where the restriction $\,a>1\,$ is sufficient and necessary.

\paragraph{(B2)} A necessary and sufficient condition for
\begin{equation*}
\left(a,\, a+ \delta, \,2a +\frac 12- \delta\right)\in \mathcal{P}_{1, 2}
\end{equation*}
is that either $\,0<\delta<\frac 12,\,0<a<\delta\,$ or $\,\delta>\frac 12,\,a>\delta.$
The special case $\,\delta = 1\,$ gives Makai's result \eqref{R2} in a sharp form.

\paragraph{(B3)} A number of interesting positivity results are obtainable
by either fixing $a$ and varying $b$ or fixing $b$ and varying $a$. For instance, we have
\begin{align*}
\left(1,\,b,\,\frac 72 -b\right) &\in \po\quad\text{if and only if}
\quad \frac 32<b<2,\\
\left(a,\,\frac 34,\,3a-\frac 14\right) &\in \po\quad\text{if and only if}
\quad \frac 14<a < \frac 38
\end{align*}
for which the former corresponds to Steinig's result \cite{St} on the positivity of Lommel's function in a sharp form.

\medskip

In the second place, if $\, \nu = b-1,\, c = a + b- 1/2\,,\,$ then our reduction
and evaluation of $S(n, \nu)$ by Saalsch\"uz's formula \eqref{G5} yield
\begin{align}\label{G13}
&\F\left(a\,; b, \, a + b -\frac 12\,; -\frac{\,x^2}{4}\right)\nonumber\\
&\qquad\quad = \J_{b-1}^2\left(\frac x2\right) +
\sum_{n=1}^\infty \frac{2n + 2b-2}{n+ 2b-2}\frac{(2b-1)_n}{n!\,(b)_n^2}\nonumber\\
&\qquad\qquad\qquad\times\,\,\frac{\left(b-a- 1/2\right)_n}{\left(a+b -1/2\right)_n}
\left(\frac x4\right)^{2n} \J^2_{n+ b -1}\left(\frac x2\right)
\end{align}
and we obtain the following positivity result by analyzing the coefficients along the same
lines of reasonings as before.

\smallskip

\begin{theorem}\label{theoremC}
For $\,a>0,\,b>0,\,$ we have
\begin{equation}\label{G14}
\left(a, \,b, \,a +b- \frac 12\right)\in\mathcal{P}_{1, 2}
\quad\text{if and only if}\quad b>a+\frac 12.
\end{equation}
Moreover, the following inequality holds true in such a case:
\begin{align}\label{G15}
\F\left(a\,; b, \, a + b-\frac 12 \,; -\frac{\,x^2}{4}\right)
>\J_{b-1}^2\left(\frac x2\right).
\end{align}
\end{theorem}

\paragraph{(B4)} We have $\,(a, \,a+1,\,3/2)\in\po\,$ if and only if $\,0<a<1\,$ with
\begin{align}\label{G16}
\F\left( a\,; a+1,\, \frac 32 \,; -\frac{\,x^2}{4}\right) &=
\frac\pi x\,\sum_{n=0}^\infty (2n +1) \frac{(1-a)_n}{(1+a)_n}
J^2_{n+\frac 12}\left(\frac x2\right)\nonumber\\
&>\left[\frac{\sin(x/2)}{x/2}\right]^2\,=\J_{\frac 12}^2\left(\frac x2\right).
\end{align}

As an application of \eqref{G16}, we obtain the positivity of the generalized sine integral (see Luke \cite{L})
in the form
\begin{align}\label{G16-1}
\int_0^x t^{-\delta} \sin t\, dt &= \frac{x^{2-\delta}}{2-\delta}\cdot \F\left(\frac{2-\delta}{2}\,;
\frac{4-\delta}{2}\,,\,\frac 32\,; -\frac{\,x^2}{4}\right)\nonumber\\
&=\frac{\pi x^{1-\delta}}{2-\delta}\,
\sum_{n=0}^\infty (2n +1) \frac{(\delta/2)_n}{(2-\delta/2)_n}
J^2_{n+\frac 12}\left(\frac x2\right)\nonumber\\
&> \frac{x^{2-\delta}}{2-\delta}\cdot \left[\frac{\sin(x/2)}{x/2}\right]^2\qquad(0<\delta<2)
\end{align}
(see Watson \cite{Wa}, pp. 152, for the case $\,\delta=1\,$).

\medskip

In the last place, if $\, \nu = b-1,\, c = 2a,\,$ then
our reduction and evaluation of $S(n, \nu)$ by Watson's formula \eqref{G6} read as
\begin{align}\label{G18}
\F\left(a\,; b, \, 2a\,; -\frac{\,x^2}{4}\right)
&=\J_{b-1}^2\left(\frac x2\right) +\sum_{n=1}^\infty \frac{2n + b-1}{n+ b-1}\frac{(b)_n}{n!\,(b)_{2n}^2}\nonumber\\
&\quad\times \,\,
\frac{\left(b-a- 1/2\right)_n}{\left(a+1/2\right)_n}
\left(\frac x4\right)^{4n} \J^2_{2n+ b -1}\left(\frac x2\right),
\end{align}
which yields the following result.

\smallskip

\begin{theorem}\label{theoremD}
For $\,a>0,\,b>0,\,$ we have
\begin{equation}\label{G19}
\left(a, \,b, \,2a\right)\in\mathcal{P}_{1, 2}\quad\text{if and only if}
\quad b>a+\frac 12
\end{equation}
and the following
inequality holds true in such a case:
\begin{align}\label{G20}
\F\left(a\,; b, \, 2a\,; -\frac{\,x^2}{4}\right)
>\J_{b-1}^2\left(\frac x2\right).
\end{align}
\end{theorem}

We should point out that this positivity result is already obtained
by the integral transform method in part (i) of Theorem \ref{theoremA}.
The only difference is the inequality
\eqref{G20} which gives a deeper insight on positivity.

\medskip

\begin{remark}\label{remarkC}
For $\,a<b<a+\frac 12\,,$ it is simple to deduce
\begin{align}
\F\left(a\,; b, \, a + b-\frac 12 \,; -\frac{\,x^2}{4}\right)
&<\J_{b-1}^2\left(\frac x2\right),\\
\F\left(a\,; b, \, 2a\,; -\frac{\,x^2}{4}\right)
&<\J_{b-1}^2\left(\frac x2\right)
\end{align}
from Theorems \ref{theoremC} and \ref{theoremD}
so that each function on the right alternates in sign and
possesses at least one zero on the interval $\,\left(0, \,2j_{b-1}\right)$.
\end{remark}

\section{Newton diagram of positivity}
Collecting all of our preceding results, it is straightforward to determine the regions of positivity
and non-positivity in the first quadrant for each fixed $\,a>0\,$
with the aid of Lemma \ref{lemmaA} and Lemma \ref{lemmaC}.

To state, we shall denote by $O_a$ the set of all points $\,(b, c)\in\R_+^2\,$ defined in terms of
interval notation as follows:
\begin{align}
O_a &= \left(a, \,a+\frac 12\right) \times\left[3a+\frac 12 -b, \,\infty\right)\nonumber\\
&\quad \cup  \left(2a, \,2a+\frac 12\right)\times \left[3a + \frac 12-b, \,a+\frac 12\right)\nonumber\\
&\quad \cup  \left[2a+\frac 12, \,\infty\right) \times\left(a, \,a+\frac 12\right)
\quad\text{if}\quad a\ge\frac 12\,,\label{N1}\\
O_a &= \left(a, \,2a\right) \times\left[3a+\frac 12 -b, \,\infty\right)\nonumber\\
&\quad \cup  \left(a+\frac 12, \,2a+\frac 12\right)\times \left[3a + \frac 12-b, \,2a\right)\nonumber\\
&\quad \cup  \left[2a+\frac 12, \,\infty\right) \times\left(a, \,2a\right)
\quad\text{if}\quad 0<a<\frac 12\,.\label{N2}
\end{align}

As it is standard, the Newton diagram associated to a finite set of plane points
$\,\bigl\{\left(\alpha_i,\,\beta_i\right) : i= 1, \cdots, N\bigr\}\,$
refers to the convex hull containing
$$\bigcup_{i=1}^N\,\bigl\{(x, y)\in\R^2 : x\ge \alpha_i,\,\,y\ge \beta_i\bigr\}\,.$$

\begin{theorem}\label{theoremF2}
For $\,a>0,\,$ put
$$\Phi(x)=\F\left(a\,; b, c\,; -\frac{\,x^2}{4}\right).$$
Let $P_a$ be the Newton diagram associated to
$\,\Lambda = \left\{\left(a+\frac 12,\,2a\right),\,\left(2a,\,a+\frac 12\right)\right\},\,$
$O_a$ the set defined as in \eqref{N1}, \eqref{N2} and $N_a$ the complement of $\,P_a\cup O_a\,$ in $\R_+^2$ so that
the decomposition $\,\R_+^2 = P_a\cup O_a\cup N_a\,$ holds.

\begin{itemize}
\item[\rm(i)] If $\,(b, c)\in P_a\,,$ then $\,\Phi\ge 0\,$ and strict positivity holds unless $(b, c)\in\Lambda.$
\item[\rm(ii)] If $\,(b, c)\in N_a\,,$ then $\Phi$ alternates in sign and possesses at least one zero on the interval
$(0, \infty).$
\end{itemize}
\end{theorem}

\smallskip

\begin{remark} For $\,(b, c)\in O_a\,,$ we do not know if $\Phi$ is positive or alternates in sign.
We include Figure \ref{Fig33} to illustrate $\,P_a, \,O_a, \,N_a\,$
in the case $\,a>\frac 12\,.$
\end{remark}

\medskip

\begin{figure}[!h]
 \centering
 \includegraphics[width=300pt, height=200pt]{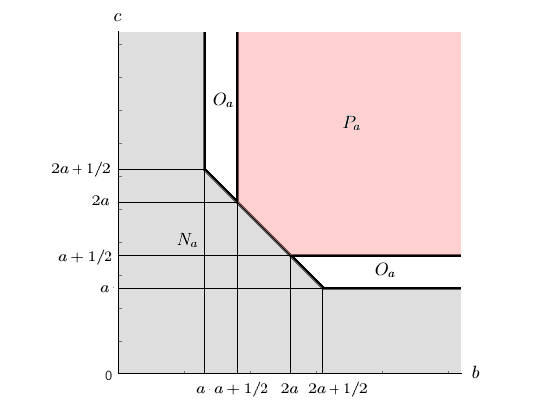}
 \caption{Newton diagram of positivity for $\,a>\frac 12$}
\label{Fig33}
\end{figure}

\medskip

As an application, we consider the problem of positivity for the Bessel integral \eqref{R0}.
By Theorem \ref{theoremB0}, the necessity region is easily seen to be
\begin{align*}
C =\left\{(\alpha, \beta): \alpha>-1,\,\,\beta\ge -\frac 12\,,\,\,-\alpha-1<\beta<\alpha+1\right\}\,.
\end{align*}
Applying Theorem \ref{theoremF2} case by case, it is straightforward to deduce
\begin{equation}
\int_0^x J_\alpha(t) t^{-\beta} dt>0\quad\text{for}\quad (\alpha, \beta)\in A = C\setminus B,
\end{equation}
as depicted in Figure \ref{Fig34}, where $B$ denotes the parallelogram given by
\begin{align*}
B =\left\{(\alpha, \beta): -1<\alpha+\beta<0, \,\,-\frac 12\le \beta<0\right\}\,.
\end{align*}

As a limiting case of certain sums of Jacobi polynomials, the positivity of \eqref{R0}
has been investigated by many authors and we refer to Askey \cite{As2} for its historical backgrounds.
We should remark that Askey also obtained the positivity region $A$ by using an interpolation
argument for which the earlier results \eqref{R1}, \eqref{R2} of Cooke  and Makai play
the roles of {\it end-points}.

\begin{figure}[!h]
 \centering
 \includegraphics[width=280pt, height=200pt]{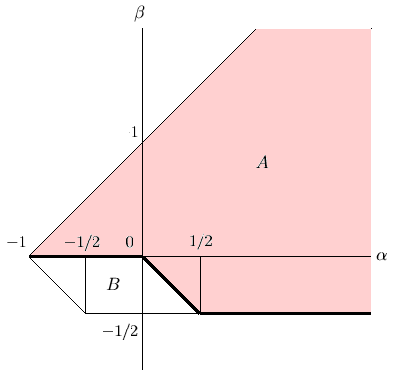}
 \caption{The region of positivity for the Bessel integral}
\label{Fig34}
\end{figure}

\section{An extension theorem}
Although it is far beyond the scope of the present work, it would be worthwhile to indicate how our
results could be extended in a simple way to the ${}_pF_{p+1}$ generalized hypergeometric functions of certain type.

\smallskip

\begin{theorem}\label{theoremH}
Suppose that $\,\alpha>0,\,\beta>0,\,\gamma>0\,$ and
$$\F\left(\alpha\,;\beta,\,\gamma\,;-\frac{\,x^2}{4}\right)\ge 0.$$
If $\,p\ge 2\,$ and $\,0<a_j<b_j\,,\,j= 1, \cdots, p-1,\,$ then
\begin{equation*}
\Phi(x)={}_pF_{p+1} \left(\begin{array}{c} \alpha,\,a_1,\cdots, a_{p-1}\\
\beta,\,\gamma,\,b_1, \cdots, b_{p-1}\end{array}\biggl| -\frac{\,x^2}{4}\right)
>0.
\end{equation*}
\end{theorem}
\begin{proof}
By an elementary computation, it is plain to represent
\begin{align*}
\Phi(x) &= \prod_{j=1}^{p-1}\frac{2}{B(a_j,\,b_j-a_j)}\int_0^1\cdots\int_0^1
\F\left(\alpha\,;\beta,\,\gamma\,;-\frac{\,x^2t_1^2\cdots t_{p-1}^2}{4}\right)\nonumber\\
&\qquad\qquad\qquad\times\quad
\left[\prod_{j=1}^{p-1}(1- t_j^2)^{b_j-a_j-1} t_j^{2a_j -1}\right] dt_1\cdots dt_{p-1}.
\end{align*}
and the result follows by the same reasonings as we have adopted before.
\end{proof}

As an illustration, let us consider
\begin{align*}
\Phi_1(x) &= {}_2F_3\left(\frac 12\,,\,1\,; \frac 34\,,\,\frac 32\,,\,2\,; -\frac{\,x^2}{4}\right),\\
\Phi_2(x) &= {}_2F_3\left(\frac 12\,,\,1\,; \frac 54\,,\,\frac 54\,,\,3\,; -\frac{\,x^2}{4}\right).
\end{align*}
Since $\,(1,\,3/2,\,2)\,$ corresponds to the square
$\,\J_{\frac 12}^2\left(\frac x2\right),\,$ we conclude $\,\Phi_1>0\,$ and Theorem \ref{theoremF2} gives $\,(1/2, \,5/4,\,5/4)\in\po\,$
so that $\,\Phi_2>0.\,$

\bigskip

\noindent
{\bf Acknowledgements.} Yong-Kum Cho is supported by National Research Foundation of Korea Grant
funded by the Korean Government (\# 20160925).
Hera Yun is supported by the Chung-Ang University
Research Scholarship Grants in 2014-2015.

\bigskip

\noindent
Yong-Kum Cho (ykcho@cau.ac.kr) and Hera Yun (herayun06@gmail.com)

\medskip

\noindent
Department of Mathematics, College of Natural Science,
Chung-Ang University, 84 Heukseok-Ro, Dongjak-Gu, Seoul 156-756, Korea

\end{document}